\documentclass[a4paper,12pt]{article}
\usepackage{hyperref}
\hypersetup{
    colorlinks=true,
    linktoc=all,
    linkcolor=red,     
    citecolor=blue,     
}
\usepackage{graphicx}
\usepackage{amsmath,amssymb,amsfonts,amsthm,graphics,enumitem}
\usepackage{latexsym, bm}
\usepackage{multicol}
\usepackage{indentfirst}

\usepackage{setspace}
\newtheorem{theorem}{Theorem}[section]
\newtheorem{lemma}{Lemma}[section]

\newtheorem{proposition}{Proposition}[section]
\newtheorem{definition}{Definition}[section]

\newtheorem{example}{Example}[section]

\newtheorem{remark}{Remark}[section]

\newtheorem{observation}{Observation}[section]
\newcommand{\ignore}[1]{}
\newtheorem{claimnum}{Claim}

\begin{document}
\noindent
\begin{spacing}{1}

\title {Rainbow Hamiltonian paths in graph collections under edge-size conditions}
\date{}

\author{
Xiaoxue Zhang\footnote{School of Mathematical Sciences, South China Normal University, Guangzhou, 510631, P. R. China.
E-mail: {\tt zhang\_xx1209@163.com}.}\;\;   \;\; Lihua You\footnote{Corresponding author. School of Mathematical Sciences, South China Normal University, Guangzhou, 510631, P. R. China.
E-mail: {\tt ylhua@scnu.edu.cn}.}
}
\maketitle
\begin{abstract}
Let $\mathbf{G}=\{G_1,\dots,G_{n-1}\}$ be a collection of not necessarily distinct graphs of order $n$ on the same vertex set $V$. A path $P$ with $V(P)\subseteq V$ is called \emph{rainbow} in $\mathbf{G}$, if there exists an injection $\phi\colon E(P)\to [n-1]$ such that $e\in E(G_{\phi(e)})$ for each $e\in E(P)$. In this paper, we show that if $n\ge3$, $e(G_i)> \binom{n-2}{2}+2$ for each $i\in[n-1]$ and $G_j$ is connected for some $j\in[n-1]$, then there exists a rainbow Hamiltonian path in $\mathbf{G}$ unless
$n=6$ with $G_1=G_2=\dots=G_{5}\cong K_2\vee 4K_1$, or $n=8$ with $G_1=G_2=\dots=G_{7}\cong K_3\vee 5K_1$. This corrects Theorem 6 in [Graphs
Combin., \textbf{41}(6)(2025), 119] and relaxes its hypothesis from requiring all $G_i$ to be connected to only requiring that at least one $G_i$ is connected, where $i\in[n-1]$. Moreover, we show that if $n\ge2$ and $e(G_i)\ge \binom{n-1}{2}$ for each $i\in[n-1]$, then $\mathbf{G}$ contains a rainbow Hamiltonian path unless $G_1=G_2=\dots=G_{n-1}\cong K_{n-1}\cup K_1$, or $n=4$ with $G_1=G_2=G_3\cong K_{1,3}$. As a consequence, Theorem 2 of [Discrete Math., \textbf{347}(2024), 114128] follows immediately.

\noindent
{\bf Keywords:} \ Rainbow Hamiltonian path; Graph collection; Edge-size condition

\noindent
{\bf MSC:} 05C38, 05C50
\end{abstract}

\section{Introduction}
A Hamiltonian path in a graph $G$ is a path that visits every vertex of $G$ exactly once. Finding sufficient conditions for the existence of a Hamiltonian path is a fundamental problem in extremal graph theory. Classical sufficient conditions include the Dirac-type condition, the Ore-type condition, the spectral radius condition, and so on. Recently, these classical results have been investigated in the context of graph collections, which can be seen as the rainbow versions of these results. Now, we give the definitions of graph collections and rainbow subgraphs.

Let $\mathbf{G}=\{G_1,\dots,G_s\}$ be a collection of not necessarily distinct graphs on the same vertex set $V$. For a graph $H$ with $V(H)\subseteq V$ and $|E(H)|\leq s$, if there exists an injection $\phi\colon E(H)\to [s]$ such that $e\in E(G_{\phi(e)})$ for each $e\in E(H)$, in other words, each edge of $H$ comes from a distinct graph $G_i$ with $i\in[s]$, then we say that $H$ is \emph{rainbow} in $\mathbf{G}$.

A great deal of research work on graph collections is available in \cite{Aharoni,Cheng2026,Cheng,Guo,Guo2026,He,Joos,Li2023,Li2024,LiandWang,Liu,MaandZhang,Sun,You,Zhang2026,Zhang}, bipartite graph collections \cite{Bradshaw,Chen,Chen2026,He2026,Hu,Ma,You2026}, directed graph collections \cite{Babinski,Chakraborti,ChengandLi,Gerbner,Li2025}, hypergraph collections \cite{Cheng2023,Cheng2025,Gao2022,Lu2022,Lu2021,Tang}, and so on. For more results, we refer the readers to the survey \cite{SunandWang}.

Several results concerning rainbow Hamiltonian paths in graph collections have been obtained.
In 2023, Li, Li and Li \cite{Li2023} proved that for a collection $\mathbf{G}=\{G_1,\dots, G_{n}\}$ of graphs of order $n$ on the same vertex set, if $\sigma_2(G_i)\ge n-1$ for each $i\in[n]$, then $\mathbf{G}$ contains a rainbow Hamiltonian path, where $\sigma_2(G_i)=\min\{d_{G_i}(u)+d_{G_i}(v): u,v\in V, u\neq v, uv\notin E(G_i)\}$. In 2024, He, Li and Feng \cite{He} investigated the spectral conditions for a collection of graphs that guarantee the existence of a rainbow Hamiltonian path.
In 2026, Cheng, Sun, Wang and Wei \cite{Cheng2026} proved that if $n\ge 2$ and $\mathbf{G}=\{G_1,\dots, G_{n-1}\}$ is a collection of graphs of order $n$ on the same vertex set such that $\delta(G_i)\geq\frac{n-1}{2}$ for each $i\in[n-1]$, then there exists a rainbow Hamiltonian path in $\mathbf{G}$.
In 2025, Liu, Zhang and Wang \cite{Liu} considered the edge-size condition for the existence of a rainbow Hamiltonian path and obtained the following result.

\begin{theorem}[\upshape\cite{Liu}]\label{thm1-1}
Let $n\ge 4$ and $\mathbf{G}=\{G_1, \dots, G_{n-1}\}$ be a collection of not necessarily distinct graphs of order $n$ on the same vertex set $V$ such that $G_i$ is connected and $e(G_i)> \binom{n-2}{2}+2$ for each $i\in [n-1]$. Then $\mathbf{G}$ contains a rainbow Hamiltonian path.
\end{theorem}

As they mentioned, Theorem \ref{thm1-1} extends the case $e(G)>\binom{n-2}{2}+2$ of the following result without characterizing its extremal graphs.

\begin{theorem}[\upshape\cite{Ning}]\label{thm1-2}
Let $G$ be a graph of order $n\ge 4$ with $\delta(G)\ge 1$. If $e(G)\ge \binom{n-2}{2}+2$, then $G$ contains a Hamiltonian path, unless $G\in \{K_1\vee (K_{n-3}\cup 2K_1),\, K_1\vee (K_{1,3}\cup K_1),\, K_{2,4},\, K_2\vee 4K_1,\, K_2\vee (3K_1\cup K_2),\, K_1\vee K_{2,5},\, K_3\vee 5K_1,\, K_2\vee (K_{1,4}\cup K_1),\, K_4 \vee 6K_1\}$.
\end{theorem}

A direct calculation shows that among all the exceptional graphs in Theorem \ref{thm1-2}, we have $e(G)>\binom{n-2}{2}+2$ for $G\in\{K_2\vee 4K_1,K_3\vee 5K_1\}$, and $e(G)=\binom{n-2}{2}+2$ for all others. This yields two counterexamples to Theorem \ref{thm1-1}: the case $n=6$ with $G_1=G_2=\dots=G_{5}\cong K_2\vee 4K_1$; the case $n=8$ with $G_1=G_2=\dots=G_{7}\cong K_3\vee 5K_1$. These two graph collections satisfy all assumptions of Theorem \ref{thm1-1}, but contain no rainbow Hamiltonian path.

Motivated by the above observations, in this paper, we correct Theorem \ref{thm1-1}, relax its hypothesis from requiring all $G_i$ to be connected to only requiring that at least one $G_i$ is connected for $i\in[n-1]$, and obtain the following result.

\begin{theorem}\label{thm1-3}
Let $n\ge 3$ and $\mathbf{G}=\{G_1,\dots,G_{n-1}\}$ be a collection of not necessarily distinct graphs of order $n$ on the same vertex set $V$. If $e(G_i)>\binom{n-2}{2}+2$ for each $i\in [n-1]$ and there exists $j\in[n-1]$ such that $G_j$ is connected, then one of the following statements holds:
\begin{enumerate}[label=(\roman*), font=\upshape, itemsep=2pt, align=left, leftmargin=1em]
 \item there exists a rainbow Hamiltonian path in $\mathbf{G}$;
 \item $n=6$ and $G_1=G_2=\dots=G_{5}\cong K_2\vee 4K_1$;
 \item $n=8$ and $G_1=G_2=\dots=G_{7}\cong K_3\vee 5K_1$.
\end{enumerate}
\end{theorem}

Furthermore, we study the edge-size condition with no connectivity restriction on the $G_i$ ($i\in [n-1]$), and prove the following result.

\begin{theorem}\label{thm1-4}
Let $n\ge 2$ and $\mathbf{G}=\{G_1,\dots,G_{n-1}\}$ be a collection of not necessarily distinct graphs of order $n$ on the same vertex set $V$. If $e(G_i)\ge \binom{n-1}{2}$ for each $i\in [n-1]$, then one of the following statements holds:
\begin{enumerate}[label=(\roman*), font=\upshape, itemsep=2pt, align=left, leftmargin=1em]
 \item there exists a rainbow Hamiltonian path in $\mathbf{G}$;
 \item $G_1=G_2=\dots=G_{n-1}\cong K_{n-1}\cup K_1$;
 \item $n=4$ and $G_1=G_2=G_{3}\cong K_{1,3}$.
\end{enumerate}
\end{theorem}

Here, we introduce some basic notation and concepts. For any terminology or notation that are not explicitly defined, we refer the reader to \cite{Bondy}.
All graphs considered in this paper are simple, finite and undirected.
Let $G=(V(G),E(G))$ be a simple graph with vertex set $V(G)$ and edge set $E(G)$, where $|E(G)|=e(G)$. For an edge $uv$ of $G$, we sometimes write $\{u,v\}$ for this edge.
For a vertex $u\in V(G)$, the set of neighbors of $u$ in $G$ is denoted by $N_{G}(u)$, $d_{G}(u)=|N_{G}(u)|$ is the \emph{degree} of  $u$ in $G$, and $N_{G}[u]=N_{G}(u)\cup \{u\}$ is the closed neighborhood of $u$ in $G$. For $u,v\in V(G)$, a \emph{$u$-$v$ walk} and a \emph{$u$-$v$ path} are a walk and a path starting at $u$ and ending at $v$, respectively. The \emph{distance} between $u$ and $v$ is the length of the shortest $u$-$v$ path in $G$, and is denoted by $d_{G}(u,v)$.
For $X\subseteq V(G)$, we use $G[X]$ to denote the subgraph of $G$ induced by $X$. For $X\subseteq V(G)$ and $Y\subseteq V(G)\backslash X$, $E_{G}[X,Y]$ denotes the set of edges of $G$ between $X$ and $Y$.

An \emph{independent set} in a graph is a set of vertices no two of which are adjacent.
The \emph{union} of two graphs $G$ and $H$ is the graph $G\cup H$ with vertex set $V(G)\cup V(H)$ and edge set $E(G)\cup E(H)$. For two vertex-disjoint graphs $G$ and $H$, the \emph{join} of $G$ and $H$ is obtained from $G\cup H$ by joining each vertex of $G$ to each vertex of $H$, denoted by $G\vee H$.
We use $C_n$, $P_n$, $K_n$ and $K_{a,b}$ to denote the cycle of order $n$, the path of order $n$, the complete graph of order $n$ and the complete bipartite graph with partitions of size $a$ and $b$, respectively. For positive integers $n$ and $a\le b$, let $[n]=\{1,2,\dots,n\}$ and $[a,b]=\{a,a+1,\dots,b\}$. For convenience, we take the vertex set $V=[n]$ in Theorems \ref{thm1-3} and \ref{thm1-4}.

The rest of this paper is organized as follows. In Section 2, we introduce an important technique and present several auxiliary results. In Section 3, we present the proofs of Theorems \ref{thm1-3} and \ref{thm1-4}. In Section 4, we conclude this paper.

\section{The Kelmans operation}
In this section, we introduce the Kelmans operation, an important shifting technique originally introduced by Kelmans \cite{Kelmans} and widely used in extremal and spectral graph theory. It is a key tool for our proofs, and we present several auxiliary results that will be used later.

\begin{definition}[\upshape\cite{Kelmans}]\label{def-1}
Let $x,y$ be two vertices of a graph $G$ on vertex set $[n]$. The Kelmans operation on $G$ from $x$ to $y$ produces a new graph $\mathcal{K}_{x,y}(G)$, which is obtained by deleting all edges between $y$ and $N_G(y)\setminus N_G[x]$, and adding all edges between $x$ and $N_G(y)\setminus N_G[x]$.
\end{definition}

\begin{definition}\label{def-2}
If $\mathcal{K}_{x,y}(G)=G$ for every pair $(x,y)$ satisfying $x<y$, then $G$ is said to be \emph{shifted}.
\end{definition}

The following observation follows immediately and we omit its proof.

\begin{observation}\label{ob1}
Let $\mathcal{K}_{x,y}(G)$ be the graph obtained from $G$ on vertex set $[n]$ by a Kelmans operation and $z\in [n]\setminus\{x,y\}$. Then
\begin{enumerate}[label=(\arabic*), font=\upshape, itemsep=2pt, align=left, leftmargin=1em]
 \item $e(\mathcal{K}_{x,y}(G))=e(G)$;
 \item if $yz\in E(\mathcal{K}_{x,y}(G))$, then $xz\in E(\mathcal{K}_{x,y}(G))$;
 \item if $xz,yz\in E(\mathcal{K}_{x,y}(G))$, then $xz,yz\in E(G)$; if $xz,yz\notin E(\mathcal{K}_{x,y}(G))$, then $xz,yz\notin E(G)$; if $xz\in E(\mathcal{K}_{x,y}(G))$ and $yz\notin E(\mathcal{K}_{x,y}(G))$, then exactly one of $xz$ and $yz$ belongs to $E(G)$;
 \item $xy\in E(\mathcal{K}_{x,y}(G))$ if and only if $xy\in E(G)$;
 \item $uv\in E(\mathcal{K}_{x,y}(G))$ if and only if $uv\in E(G)$, where $u,v\in [n]\setminus\{x,y\}$;
 \item if $G$ is a shifted graph, then for any $\{x_1,x_2,y_1,y_2\}\subseteq[n]$ such that $x_1\neq x_2$ and $x_i\le y_i$ for $1\le i\le 2$, $y_1y_2\in E(G)$ implies $x_1x_2\in E(G)$.
 \item iterating the Kelmans operation on pairs $(x,y)$ of vertices in $[n]$ with $x<y$ eventually yields a shifted graph.
\end{enumerate}
\end{observation}

By (2) of Observation \ref{ob1}, we easily get the following property of shifted graphs.

\begin{proposition}
If $G$ is a shifted graph on vertex set $[n]$, then $d_{G}(1)\ge d_{G}(2)\ge \dots \ge d_{G}(n)$.
\end{proposition}

\begin{proof}
Since $G$ is a shifted graph, by Definition \ref{def-2}, we have $\mathcal{K}_{x,y}(G)=G$ for any $x,y\in[n]$ with $x<y$. By (2) of Observation \ref{ob1}, $N_{G}(y)\setminus\{x\}\subseteq N_{G}(x)\setminus\{y\}$, which implies $d_{G}(y)\leq d_{G}(x)$. Therefore, $d_{G}(1)\ge d_{G}(2)\ge \dots \ge d_{G}(n)$.
\end{proof}

By (7) of Observation \ref{ob1}, every graph admits a shifted graph obtained via the Kelmans operations. We next point out an important caveat concerning shifted graphs.

For a graph $G$ and $x,y\in V(G)$, if $d_G(x,y)=1$, we often write $xy\in E(G)$. To avoid ambiguity, we sometimes write $\{x,y\}\in E(G)$.

\begin{remark}
It is worth noting that a connected graph $G$ may be disconnected by a Kelmans operation, and the shifted graph may not be unique, as it depends on the order in which we perform the Kelmans operations. We illustrate this with the following two examples.
\begin{example}
Let $n\ge 7$ and $G\cong P_n$ be the path on vertex set $[n]$ with $E(G)=\{\{i,i+1\}:1\le i\le n-1\}$.
\begin{itemize}[leftmargin=*]
\item Applying $\mathcal{K}_{1,2}$, $\mathcal{K}_{1,3}$, \dots, $\mathcal{K}_{1,n-1}$ in this order yields the star $K_{1,n-1}$ with edge set $\{\{1,i\}: 2\le i\le n\}$.
\item Applying $\mathcal{K}_{1,n}$, $\mathcal{K}_{1,{n-1}}$, \dots, $\mathcal{K}_{1,4}$ in this order yields the disconnected graph $(K_1\vee (K_2\cup(n-4)K_1))\cup K_1$ with edge set $\{\{1,i\}: 2\le i\le n-1\}\cup \{\{2,3\}\}$.

\item Applying $\mathcal{K}_{2,n}$, $\mathcal{K}_{4,5}$, \dots, $\mathcal{K}_{4,{n-2}}$, $\mathcal{K}_{1,4}$ and $\mathcal{K}_{4,{n-1}}$ in this order yields the disconnected graph $(K_1\vee (P_3\cup(n-6)K_1))\cup 2K_1$ with edge set $\{\{1,i\}: 2\le i\le n-2\}\cup \{\{2,3\},\{2,4\}\}$.
\end{itemize}

Clearly, all three resulting graphs are shifted by direct checking. The first one is connected and the other two are disconnected. Moreover, they are pairwise non-isomorphic.
\end{example}

\begin{example} Let $n\ge 5$ and $G\cong K_1\vee P_{n-1}$ be a graph on vertex set $[n]$ with $E(G)=\{\{1,i\}:2\le i\le n\}\cup\{\{i,i+1\}:2\le i\le n-1\}$.

\begin{itemize}[leftmargin=*]
\item Applying $\mathcal{K}_{2,3}$, $\mathcal{K}_{2,4}$, \dots, $\mathcal{K}_{2,n-1}$ in this order yields the graph $K_1\vee K_{1,n-2}$ with edge set $\{\{1,i\}: 2\le i\le n\}\cup \{\{2,i\}: 3\le i\le n\}$.
\item Applying $\mathcal{K}_{2,n}$, $\mathcal{K}_{2,{n-1}}$, \dots, $\mathcal{K}_{2,5}$ in this order yields the graph $K_1\vee((K_1\vee (K_2\cup(n-5)K_1))\cup K_1)$ with edge set $\{\{1,i\}: 2\le i\le n\}\cup\{\{2,i\}: 3\le i\le n-1\}\cup \{\{3,4\}\}$.
\end{itemize}

Clearly, both resulting graphs are shifted and connected, yet they are non-isomorphic by direct checking.
\end{example}
\end{remark}


We next present a lemma showing that there exists an order of performing the Kelmans operations such that a connected graph remains connected after each Kelmans operation. To this end, we first establish the following result.

\begin{lemma}\label{lem2-1}
Let $G$ be a connected graph of order $n\ge 2$, and let $x,y\in V(G)$. Then $\mathcal{K}_{x,y}(G)$ is connected if and only if $d_{G}(x,y)\le 2$.
\end{lemma}

\begin{proof}
Suppose that $\mathcal{K}_{x,y}(G)$ is connected. If $d_{G}(x,y)\ge 3$, then $xy\notin E(G)$ and $N_{G}(x)\cap N_{G}(y)=\emptyset$. It follows that $d_{\mathcal{K}_{x,y}(G)}(y)=0$, which implies that $\mathcal{K}_{x,y}(G)$ is not connected, a contradiction. Thus $d_{G}(x,y)\le 2$.

Suppose that $d_{G}(x,y)\le 2$. Then either $xy\in E(G)$ or $N_{G}(x)\cap N_{G}(y)\neq\emptyset$.
If $xy\in E(G)$, then $xy\in E(\mathcal{K}_{x,y}(G))$, and for each $z\in N_{G}(y)\setminus N_{G}[x]$, $yz\notin E(\mathcal{K}_{x,y}(G))$ and $yxz$ is a path in $\mathcal{K}_{x,y}(G)$. If $N_{G}(x)\cap N_{G}(y)\neq\emptyset$, we take $w\in N_{G}(x)\cap N_{G}(y)$, then $xw,yw\in E(\mathcal{K}_{x,y}(G))$, and for each $z\in N_{G}(y)\setminus N_{G}[x]$, $yz\notin E(\mathcal{K}_{x,y}(G))$ and $ywxz$ is a path in $\mathcal{K}_{x,y}(G)$.

For any two vertices $u,v\in V(\mathcal{K}_{x,y}(G))$, there exists a $u$-$v$ path $P$ in $G$ since $G$ is connected. If $yz\notin E(P)$ for each $z\in N_{G}(y)\setminus N_{G}[x]$, then $P$ is also a $u$-$v$ path in $\mathcal{K}_{x,y}(G)$. If $yz\in E(P)$ for some $z\in N_{G}(y)\setminus N_{G}[x]$, then there are at most two such edges on $P$. We replace each such edge $yz$ with either $yxz$ (if $xy\in E(G)$) or $ywxz$ (if $N_{G}(x)\cap N_{G}(y)\neq\emptyset$), this yields a $u$-$v$ walk in $\mathcal{K}_{x,y}(G)$, which further contains a $u$-$v$ path. Thus $\mathcal{K}_{x,y}(G)$ is connected.
\end{proof}

\begin{lemma}\label{lem2-2}
Let $G$ be a connected graph on vertex set $[n]$, where $n\ge 2$. Then there exists a finite sequence of the Kelmans operations that transforms $G$ into a connected shifted graph with vertex $1$ of degree $n-1$, and every graph in this sequence is connected.

\end{lemma}

\begin{proof}
We proceed by distinguishing two cases based on $d_G(1)$.

\vspace{6pt}
$\mathbf{Case~1.}$ $d_G(1)=n-1$.
\vspace{6pt}

If $G$ is shifted, then we are done.

If $G$ is not shifted, then we apply the Kelmans operations with $x_i<y_i$ on $G$ to generate the sequence $G=G^{(0)},G^{(1)},G^{(2)},\dots,G^{(t)}$ such that $G^{(i)}=\mathcal{K}_{x_i,y_i}(G^{(i-1)})$ for each $1\le i\le t$, which terminates at a shifted graph $G^{(t)}$.

By Definition \ref{def-1}, $d_{G^{(i)}}(1)=n-1$ for each $1\le i\le t$ since ${G}^{(i)}=\mathcal{K}_{x_i,y_i}(G^{(i-1)})$ with $x_i<y_i$, then $d_{G^{(i-1)}}(x_i,y_i)\le 2$ for each $1\le i\le t$, and thus $G^{(i)}$ is connected for each $1\le i\le t$ by Lemma \ref{lem2-1}.

\vspace{6pt}
$\mathbf{Case~2.}$ $d_G(1)<n-1$.
\vspace{6pt}

Since $d_{G}(1)< n-1$ and $G$ is connected, there is a vertex $v$ satisfying $d_{G}(1,v)=2$. Without loss of generality, we assume $\{1,u\},\{u,v\}\in E(G)$.
The Kelmans operation on $G$ from $1$ to $u$ produces a new graph $\mathcal{K}_{1,u}(G)$, which we denote by $G^{(1)}$.
Then $\{1,v\}\in E(G^{(1)})$ and $\{1,z\}\in E(G^{(1)})$ for each $z\in N_{G}(1)$, which implies that the degree of vertex $1$ strictly increases. Since $d_{G}(1,u)=1$, by Lemma \ref{lem2-1}, $G^{(1)}$ is connected.

If $d_{G^{(1)}}(1)< n-1$, then we repeat the above procedure to obtain a graph $G^{(2)}$ satisfying $d_{G^{(2)}}(1)>d_{G^{(1)}}(1)$, and $G^{(2)}$ is connected.

Iterating this procedure, and noting that each Kelmans operation strictly increases the degree of vertex $1$, this process terminates after at most $n-1-d_{G}(1)$ Kelmans operations.

Therefore, we obtain a sequence $G=G^{(0)},G^{(1)},\dots,G^{(s)}$ such that $G^{(i)}=\mathcal{K}_{x_i,y_i}(G^{(i-1)})$ (in fact, $x_i=1$) with $d_{G^{(i)}}(1)>d_{G^{(i-1)}}(1)$ for each $1\le i\le s$, and $d_{G^{(s)}}(1)=n-1$.

We continue applying the Kelmans operations with $x_i<y_i$ on $G^{(s)}$ to generate the sequence $G^{(s)},G^{(s+1)},$ $\dots,G^{(t)}$, which terminates at a shifted graph $G^{(t)}$.

Similar to Case 1, if $t=s$, then we are done; if $t>s$, by Definition \ref{def-1}, $d_{G^{(i)}}(1)=n-1$ for each $s+1\le i\le t$, then $d_{G^{(i-1)}}(x_i,y_i)\le 2$ for each $s+1\le i\le t$, and thus $G^{(i)}$ is connected by Lemma \ref{lem2-1}.

This completes the proof.
\end{proof}

The above discussion concerns the Kelmans operations applied to a single graph. Similarly, the Kelmans operations can be applied to graph collections.

\begin{definition}\label{def-3}
Let $\mathbf{G}=\{G_1,\dots,G_{n-1}\}$ be a collection of graphs on the same vertex set $[n]$ and $x,y\in[n]$. The Kelmans operation on $\mathbf{G}$ from $x$ to $y$ produces a new graph collection $\mathcal{K}_{x,y}(\mathbf{G})$, which is obtained by performing the Kelmans operation from $x$ to $y$ on every $G_i\in\mathbf{G}$, i.e., $\mathcal{K}_{x,y}(\mathbf{G})=\{\mathcal{K}_{x,y}(G_1),\dots,\mathcal{K}_{x,y}(G_{n-1})\}$.
\end{definition}

\begin{definition}\label{def-4}
If $\mathcal{K}_{x,y}(\mathbf{G})=\mathbf{G}$ for every pair $(x,y)$ satisfying $x<y$, then $\mathbf{G}$ is said to be \emph{shifted}.
\end{definition}

In fact, a graph collection $\mathbf{G}$ is shifted if and only if every $G_i\in\mathbf{G}$ is shifted. Moreover, any graph collection can be transformed into a shifted graph collection via finitely many Kelmans operations.

In 2024, He, Li and Feng \cite{He} applied the Kelmans operations to graph collections and obtained the following result.

\begin{theorem}[\upshape\cite{He}]\label{lem2-3}
Let $\mathbf{G}=\{G_1, \dots, G_{n-1}\}$ be a collection of graphs on the same vertex set $[n]$, and let $x<y$ be two vertices in $[n]$. If $\mathcal{K}_{x,y}(\mathbf{G})$ contains a rainbow Hamiltonian path, then so does $\mathbf{G}$.
\end{theorem}

Let $\mathbf{G}=\{G_1,\dots,G_{n-1}\}$ be a collection of graphs on the same vertex set $[n]$. For any finite sequence
\begin{equation*}\label{e1}
\mathbf{G}=\mathbf{G}^{(0)},\mathbf{G}^{(1)}=\mathcal{K}_{x_1,y_1}(\mathbf{G}^{(0)}),\mathbf{G}^{(2)}=\mathcal{K}_{x_2,y_2}(\mathbf{G}^{(1)}),\dots,\mathbf{G}^{(t)}=\mathcal{K}_{x_t,y_t}(\mathbf{G}^{(t-1)}),
\end{equation*}
where $x_s,y_s\in[n]$ for each $1\le s\le t$, we write $\mathbf{G}^{(s)}=\{G_1^{(s)},G_2^{(s)},\dots,G_{n-1}^{(s)}\}$ for each integer $s$ with $0\le s\le t$, and $G_i^{(s)}$ denotes the $i$-th graph in the graph collection $\mathbf{G}^{(s)}$. In particular, $G_i^{(0)}=G_i$, and $G_i^{(s)}=\mathcal{K}_{x_s,y_s}\big(G_i^{(s-1)}\big)$ for $1\le s\le t$.

\begin{lemma}\label{lem2-4}
Let $\mathbf{G}=\{G_1, \dots, G_{n-1}\}$ be a collection of graphs on the same vertex set $[n]$, and suppose that $G_1$ is connected. Then there exists a nonnegative integer $t$ such that $\mathbf{G}^{(t)}$ is shifted and $d_{G_1^{(t)}}(1)=n-1$.
\end{lemma}

\begin{proof}
By Lemma \ref{lem2-2}, there exists a finite sequence
\[
G_1=G_1^{(0)},G_1^{(1)}=\mathcal{K}_{x_1,y_1}(G_1^{(0)}),\dots,G_1^{(k)}=\mathcal{K}_{x_k,y_k}(G_1^{(k-1)}),
\]
where $x_i,y_i\in [n]$ and $x_i<y_i$ for each $1\le i\le k$, such that
$G_1^{(k)}$ is shifted, vertex $1$ has degree $n-1$ in $G_1^{(k)}$ and $G_1^{(i)}$ is connected for each $1\le i\le k$. Let $\mathbf{G}=\mathbf{G}^{(0)},\mathbf{G}^{(1)}=\mathcal{K}_{x_1,y_1}(\mathbf{G}^{(0)}),\dots,\mathbf{G}^{(k)}=\mathcal{K}_{x_k,y_k}(\mathbf{G}^{(k-1)})$.

If every graph in $\mathbf{G}^{(k)}$ is shifted, then by setting $t=k$, we are done.

If there exists a graph in $\mathbf{G}^{(k)}$ that is not shifted, then we continue applying the Kelmans operations with $x_i<y_i$ on $\mathbf{G}^{(k)}$ to obtain $\mathbf{G}^{(k+1)},\dots,\mathbf{G}^{(t)}$, such that every graph in $\mathbf{G}^{(t)}$ is shifted. Since $G_1^{(k)}$ is shifted and $d_{G_1^{(k)}}(1)=n-1$, by Definition \ref{def-2}, we have $G_1^{(i)}=G_1^{(k)}$ and $d_{G_1^{(i)}}(1)=n-1$ for each $k+1\le i\le t$, and thus $G_1^{(i)}$ is connected by Lemma \ref{lem2-1}.

This completes the proof of Lemma \ref{lem2-4}.
\end{proof}

\section{Proofs of Theorems \ref{thm1-3} and \ref{thm1-4}}
In this section, we show Theorems \ref{thm1-3} and \ref{thm1-4}. To this end, we first present some lemmas.

\begin{lemma}\label{lem3-2}
Let $G$ be a shifted graph on vertex set $[n]$, and let $e_j=\{j,n+1-j\}$ for $2\le j\le \left \lfloor \frac{n}{2}  \right \rfloor$.
\begin{enumerate}[label=(\roman*), font=\upshape, itemsep=2pt, align=left, leftmargin=1em]
 \item If $n\ge 4$ and $e(G)>\binom{n-2}{2}+2$, then one of the following holds:
 \begin{enumerate}[label=(\alph*), font=\upshape, itemsep=2pt, topsep=0pt]
    \item $e_j\in E(G)$ for each $2\le j\le \left \lfloor \frac{n}{2}  \right \rfloor$;
    \item $n=6$ and $G\cong K_{2}\vee 4K_1$, where $G[\{1,2\}]\cong K_2$ and $G[\{3,4,5,6\}]\cong 4K_1$;
    \item $n=8$ and $G\cong K_{3}\vee 5K_1$, where $G[\{1,2,3\}]\cong K_3$ and $G[\{4,5,6,7,8\}]\cong 5K_1$;
\end{enumerate}
 \item If $n\ge 5$ and $e(G)\ge\binom{n-1}{2}$, then $e_j\in E(G)$ for each $2\le j\le \left \lfloor \frac{n}{2}  \right \rfloor$.
\end{enumerate}
\end{lemma}

\begin{proof}
Suppose that there exists some $2\le j\le \left \lfloor \frac{n}{2}  \right \rfloor$ such that $e_j\notin E(G)$. We partition $V(G)$ into three subsets: $V_1=[j-1]$, $V_2=[j,n-j]$, and $V_3=[n+1-j,n]$.

First, we show $E_{G}[V_2,V_3]=\emptyset$. If not, we take $\{u,v\}\in E(G)$ for
$u\in V_2$ and $v\in V_3$. By (6) of Observation \ref{ob1} and setting $x_1=j$, $y_1=u$, $x_2=n+1-j$, $y_2=v$, we have $e_j=\{j,n+1-j\}\in E(G)$ as $\{u,v\}\in E(G)$, a contradiction.

Next, we show that $V_3$ is an independent set in $G$. Otherwise, there exist $u,v\in V_3$ such that $\{u,v\}\in E(G)$. By (6) of Observation \ref{ob1} and setting $x_1=j$, $y_1=u$, $x_2=y_2=v$, we have $\{j,v\}\in E(G)$ as $\{u,v\}\in E(G)$, which implies $E_{G}[V_2,V_3]\neq\emptyset$, a contradiction.

Therefore, we have
\begin{align}
e(G)&\le \binom{n}{2}-|V_2||V_3|-\binom{|V_3|}{2} \notag\\
&=\binom{n}{2}-(n+1-2j)j-\binom{j}{2} \notag\\
&=\binom{n-2}{2}+2+\frac{(j-2)(3j+5-2n)}{2} \label{e2}\\
&=\binom{n-1}{2}-\frac{(j-1)(2n-3j-2)}{2}. \label{e3}
\end{align}

We first show that (i) holds. Since (a) does not hold by assumption, it suffices to show that (b) or (c) holds.
Combining \eqref{e2} with $e(G)>\binom{n-2}{2}+2$, we obtain $\frac{(j-2)(3j+5-2n)}{2}>0$. Clearly, $j\neq2$. If $j\ge 5$, then $3j+5-2n\le 5-j\le 0$ as $n\ge 2j$, and thus $\frac{(j-2)(3j+5-2n)}{2}\le 0$, a contradiction.
Therefore, $(n,j)=(6,3)$ or $(n,j)=(8,4)$ since $\frac{(j-2)(3j+5-2n)}{2}>0$ and $j\le \left \lfloor \frac{n}{2}  \right \rfloor$.

Furthermore, if $(n,j)=(6,3)$ or $(n,j)=(8,4)$, by \eqref{e2}, we have  $e(G)=\binom{n}{2}-|V_2||V_3|-\binom{|V_3|}{2}$ since $e(G)>\binom{n-2}{2}+2$, which implies $G\cong K_{2}\vee 4K_1$ with $G[\{1,2\}]\cong K_2$ and $G[\{3,4,5,6\}]\cong 4K_1$, or $G\cong K_{3}\vee 5K_1$ with $G[\{1,2,3\}]\cong K_3$ and $G[\{4,5,6,7,8\}]\cong 5K_1$. Therefore, (b) or (c) holds, and thus (i) holds.

Now we show that (ii) holds. If $j=2$, then $2n-3j-2>0$ as $n\ge 5$; if $j\ge3$, then $2n-3j-2\ge j-2>0$ as $n\ge 2j$. Thus $e(G)<\binom{n-1}{2}$ by \eqref{e3}, contradicting $e(G)\ge\binom{n-1}{2}$. Therefore, $e_j\in E(G)$ for each $2\le j\le \left \lfloor \frac{n}{2}  \right \rfloor$, (ii) holds.

This completes the proof.
\end{proof}

\begin{lemma}\label{lem3-3}
Let $r\ge 2$, $n=2r+2$, and let $\mathbf{G}=\{G_1, \dots, G_{n-1}\}$ be a collection of graphs on the same vertex set $[n]$ that contains no rainbow Hamiltonian path. If $\mathcal{K}_{x,y}(G_1)=\mathcal{K}_{x,y}(G_2)=\dots=\mathcal{K}_{x,y}(G_{n-1})\cong K_r\vee (n-r)K_1$ for two distinct vertices $x,y\in [n]$, then $G_1=G_2=\dots=G_{n-1}\cong K_r\vee (n-r)K_1$.
\end{lemma}

\begin{proof}
Without loss of generality, let $A\subseteq [n]$ be the set of $r$ vertices of degree $n-1$ in $\mathcal{K}_{x,y}(G_i)$ for each $i\in[n-1]$, and let $B=[n]\setminus A$. Then $|B|=r+2$.

First, we show $x,y\in A$, or $x,y\in B$, or $x\in A$ with $y\in B$. Otherwise, $x\in B$ and $y\in A$, thus for each $i\in[n-1]$, $yz\in E(\mathcal{K}_{x,y}(G_i))$ and $xz\notin E(\mathcal{K}_{x,y}(G_i))$, where $z\in B\setminus\{x\}$. This contradicts (2) of Observation \ref{ob1}.

Now we proceed by distinguishing the following two cases.

\vspace{6pt}
$\mathbf{Case~1.}$ $x,y\in A$ or $x,y\in B$.
\vspace{6pt}

By (3)-(5) of Observation \ref{ob1}, we readily obtain $G_i=\mathcal{K}_{x,y}(G_i)$ for each $i\in[n-1]$, and the conclusion follows.

\vspace{6pt}
$\mathbf{Case~2.}$ $x\in A$ and $y\in B$.
\vspace{6pt}

We set $A'=A\setminus\{x\}=\{u_1,u_2,\dots,u_{r-1}\}$ and $B'=B\setminus\{y\}=\{v_1,v_2,\dots,v_{r+1}\}$.
By (3)-(5) of Observation \ref{ob1}, for each $i\in[n-1]$, we have $xy\in E(G_i)$, $d_{G_i}(u)=n-1$ for each $u\in A'$, $B'$ is an independent set in $G_i$, and exactly one of $xv$ and $yv$ belongs to $E(G_i)$ for each $v\in B'$.

Now we show that $xv\in E(G_i)$ for each $v\in B'$ and each $i\in [n-1]$, or $yv\in E(G_i)$ for each $v\in B'$ and each $i\in [n-1]$. If not, there exist distinct $p,q\in[n-1]$ and distinct $v_a,v_b\in B'$ such that $xv_a\in E(G_p)$ and $yv_b\in E(G_q)$. For convenience, let $B'\setminus\{v_a,v_b\}=\{v'_1,v'_2,\dots,v'_{r-1}\}$. Recall that for each $i\in [n-1]$, $xy\in E(G_i)$ and $d_{G_i}(u)=n-1$ for each $u\in A'$.
Then $P=v_axyv_bu_1v'_1u_2v'_2\dots u_{r-1}v'_{r-1}$ is a rainbow Hamiltonian path in $\mathbf{G}$, where $xv_a\in E(G_p)$, $yv_b\in E(G_q)$ and $E(P)\setminus\{xv_a,yv_b\}\subseteq E(G_i)$ for each $i\in[n-1]$, a contradiction.

If $xv\in E(G_i)$ for each $v\in B'$ and each $i\in [n-1]$, then $G_i=\mathcal{K}_{x,y}(G_i)$ for each $i\in [n-1]$. 
If $yv\in E(G_i)$ for each $v\in B'$ and each $i\in [n-1]$, then $G_i\cong K_r\vee (n-r)K_1$ with $G_i[A'\cup\{y\}]\cong K_r$ and $G_i[B'\cup\{x\}]\cong (n-r)K_1$ for each $i\in [n-1]$.
Therefore, $G_1=G_2=\dots=G_{n-1}\cong K_r\vee (n-r)K_1$.

This completes the proof of Lemma \ref{lem3-3} by Cases 1 and 2.
\end{proof}

The proofs of the following two lemmas are similar to that of Lemma \ref{lem3-3}. We omit them here and defer the details to Appendix A.

\begin{lemma}\label{lem3-4}
Let $n\ge 4$ and $\mathbf{G}=\{G_1, \dots, G_{n-1}\}$ be a collection of graphs on the same vertex set $[n]$ that contains no rainbow Hamiltonian path. If $\mathcal{K}_{x,y}(G_1)=\mathcal{K}_{x,y}(G_2)=\dots=\mathcal{K}_{x,y}(G_{n-1})\cong K_{n-1}\cup K_1$ for two distinct vertices $x,y\in [n]$, then $G_1=G_2=\dots=G_{n-1}\cong K_{n-1}\cup K_1$.
\end{lemma}
\begin{proof}
See Appendix A.1.
\end{proof}

\begin{lemma}\label{lem3-5}
Let $\mathbf{G}=\{G_1, G_2, G_3\}$ be a collection of graphs on the same vertex set $[4]$ that contains no rainbow Hamiltonian path. If $\mathcal{K}_{x,y}(G_1)=\mathcal{K}_{x,y}(G_2)=\mathcal{K}_{x,y}(G_3)\cong K_{1,3}$ for two distinct vertices $x,y\in [4]$, then $G_1=G_2=G_3\cong K_{1,3}$.
\end{lemma}
\begin{proof}
See Appendix A.2.
\end{proof}

Now we give the proofs of Theorems \ref{thm1-3} and \ref{thm1-4}.
\begin{proof}[\noindent\textbf{Proof of Theorem \ref{thm1-3}}]
If $n=3$, then $e(G_1)=e(G_2)=3$ since $e(G_i)>\binom{n-2}{2}+2$ for $i\in[2]$, and thus $G_1=G_2\cong K_3$, which implies that (i) holds. Therefore, we only consider $n\ge 4$ in the rest of the proof. We assume that (i) does not hold, now we show that either (ii) or (iii) holds. Without loss of generality, we assume that $G_1$ is connected.

By Lemma \ref{lem2-4} and its proof, there exists a finite sequence
\begin{equation}\label{e4}
\mathbf{G}=\mathbf{G}^{(0)},\mathbf{G}^{(1)}=\mathcal{K}_{x_1,y_1}(\mathbf{G}^{(0)}),\dots,\mathbf{G}^{(t)}=\mathcal{K}_{x_t,y_t}(\mathbf{G}^{(t-1)}),
\end{equation}
where $x_s,y_s\in [n]$ and $x_s<y_s$ for each $1\le s\le t$, such that $\mathbf{G}^{(t)}=\{G_1^{(t)},G_2^{(t)},\dots,G_{n-1}^{(t)}\}$ is shifted and $d_{G_1^{(t)}}(1)=n-1$, which implies that $G_i^{(t)}$ is shifted for each $i\in[n-1]$.
It follows from (1) of Observation \ref{ob1} that $e(G_i^{(t)})=e(G_i)>\binom{n-2}{2}+2$ for each $i\in [n-1]$.
Since $\mathbf{G}$ contains no rainbow Hamiltonian path, by Theorem \ref{lem2-3}, neither does $\mathbf{G}^{(s)}$ for each $1\le s\le t$.

Let $e_1=\{1,n\}$, $e_j=\{j,n+1-j\}$ for $2\le j\le \left \lfloor \frac{n}{2}  \right \rfloor$ and $e'_j=\{j,n-j\}$ for $1\le j\le \left \lfloor \frac{n-1}{2}  \right \rfloor$. Then $e_1\in E(G_1^{(t)})$ as $d_{G_1^{(t)}}(1)=n-1$.

\begin{claimnum}\label{claim1}
$n\in\{6,8\}$.
\end{claimnum}

\begin{proof}
If $n\notin\{6,8\}$, then $e_j\in E(G_i^{(t)})$ for each $2\le j\le \left \lfloor \frac{n}{2}  \right \rfloor$ and each $i\in [n-1]$ by (i) of Lemma \ref{lem3-2}.
For each $i\in [n-1]$, since $e_2=\{2,n-1\}\in E(G_i^{(t)})$, by (6) of Observation \ref{ob1}, we have $e'_1=\{1,n-1\}\in E(G_i^{(t)})$ by setting $x_1=1$, $y_1=2$ and $x_2=y_2=n-1$. For each $i\in [n-1]$ and each $2\le j\le \left \lfloor \frac{n-1}{2}  \right \rfloor$, since $e_j=\{j,n+1-j\}\in E(G_i^{(t)})$, we have $e'_j=\{j,n-j\}\in E(G_i^{(t)})$ by (6) of Observation \ref{ob1} and setting $x_1=y_1=j$, $x_2=n-j$ and $y_2=n+1-j$.


Therefore, for odd $n$, $e_1e'_1e_2e'_2\cdots e_{\frac{n-1}{2}}e'_{\frac{n-1}{2}}$ is a rainbow Hamiltonian path in $\mathbf{G}^{(t)}$; for even $n$, $e_1e'_1e_2e'_2\cdots e_{\frac{n}{2}-1}e'_{\frac{n}{2}-1}e_{\frac{n}{2}}$ is a rainbow Hamiltonian path in $\mathbf{G}^{(t)}$. This contradicts that $\mathbf{G}^{(t)}$ contains no rainbow Hamiltonian path, and thus $n\in\{6,8\}$.
\end{proof}

By Claim \ref{claim1}, we let $n=2r+2$, where $r\in\{2,3\}$. For each $i\in [n-1]$, we have $e_j\in E(G_i^{(t)})$ for each $2\le j\le r$ by (i) of Lemma \ref{lem3-2}, and thus $e'_j\in E(G_i^{(t)})$ for each $1\le j\le r$ by (6) of Observation \ref{ob1}. Now we only need to check whether $e_{r+1}=\{r+1,n-r\}\in E(G_i^{(t)})$ for each $i\in [n-1]$.

\begin{claimnum}\label{claim2}
$e_{r+1}=\{r+1,n-r\}\notin E(G_i^{(t)})$ for each $i\in [n-1]$.
\end{claimnum}
\begin{proof}
If $e_{r+1}\in E(G_k^{(t)})$ for some $2\le k\le n-1$, then $e_1e'_1e_2e'_2\cdots e_{r}e'_{r}e_{r+1}$ is a rainbow Hamiltonian path in $\mathbf{G}^{(t)}$, where $e_1\in E(G_1^{(t)})$ and $e_{r+1}\in E(G_k^{(t)})$, a contradiction.

If $e_{r+1}\notin E(G_k^{(t)})$ for each $2\le k\le n-1$ and $e_{r+1}\in E(G_1^{(t)})$, then $G_k^{(t)}\cong K_{r}\vee (n-r)K_1$ with $G_k^{(t)}[\{1,2,\dots,r\}]\cong K_r$ and $G_k^{(t)}[\{r+1,r+2,\dots,n\}]\cong (n-r)K_1$ for each $2\le k\le n-1$ by (b),(c) of Lemma \ref{lem3-2}, and thus $e_1\in E(G_k^{(t)})$. It follows that $e_1e'_1e_2e'_2\cdots e_{r}e'_{r}e_{r+1}$ is a rainbow Hamiltonian path in $\mathbf{G}^{(t)}$, where $e_1\in E(G_2^{(t)})$ and $e_{r+1}\in E(G_1^{(t)})$, a contradiction.

Therefore, $e_{r+1}\notin E(G_i^{(t)})$ for each $i\in [n-1]$.
\end{proof}

By Claims \ref{claim1} and \ref{claim2}, we have (b) or (c) of Lemma \ref{lem3-2} holds, say, $G_i^{(t)}\cong K_{r}\vee (n-r)K_1$ with $G_i^{(t)}[\{1,2,\dots,r\}]\cong K_r$ and $G_i^{(t)}[\{r+1,r+2,\dots,n\}]\cong (n-r)K_1$ for each $i\in [n-1]$, then $G_1=G_2=\dots=G_{n-1}\cong K_r\vee (n-r)K_1$ by \eqref{e4} and Lemma \ref{lem3-3}, and thus either (ii) or (iii) holds since $n=2r+2$ with $r\in\{2,3\}$.

Therefore, the proof of Theorem \ref{thm1-3} is completed.
\end{proof}

\begin{proof}[\noindent\textbf{Proof of Theorem \ref{thm1-4}}]
For $n=2$, the theorem is trivial. For $n=3$, it is straightforward to verify that either $\mathbf{G}$ contains a rainbow Hamiltonian path or $G_1=G_2\cong K_2\cup K_1$. Therefore, we only consider $n\ge 4$ in the rest of the proof. 

We assume that (i) does not hold, now we show that either (ii) or (iii) holds.

Clearly, there exists a finite sequence
\begin{equation}\label{e5}
\mathbf{G}=\mathbf{G}^{(0)},\mathbf{G}^{(1)}=\mathcal{K}_{x_1,y_1}(\mathbf{G}^{(0)}),\dots,\mathbf{G}^{(t)}=\mathcal{K}_{x_t,y_t}(\mathbf{G}^{(t-1)}),
\end{equation}
where $x_s,y_s\in [n]$ and $x_s<y_s$ for each $1\le s\le t$, such that $\mathbf{G}^{(t)}=\{G_1^{(t)},G_2^{(t)},\dots,G_{n-1}^{(t)}\}$ is shifted.
By (1) of Observation \ref{ob1}, $e(G_i^{(t)})=e(G_i)\ge \binom{n-1}{2}$ for each $i\in [n-1]$.
By Theorem \ref{lem2-3}, $\mathbf{G}^{(s)}$ contains no rainbow Hamiltonian path for each $1\le s\le t$.

Now we show that either (I) $G_1^{(t)}=G_2^{(t)}=\dots=G_{n-1}^{(t)}\cong K_{n-1}\cup K_1$, or (II) $n=4$ with $G_1^{(t)}=G_2^{(t)}=G_3^{(t)}\cong K_{1,3}$ by considering the following two cases.

\vspace{6pt}
$\mathbf{Case~1.}$ $n\ge 5$.
\vspace{6pt}

Let $e_j=\{j,n+1-j\}$ for $1\le j\le \left \lfloor \frac{n}{2}  \right \rfloor$ and $e'_j=\{j,n-j\}$ for $1\le j\le \left \lfloor \frac{n-1}{2}  \right \rfloor$. Then $e_j\in E(G_i^{(t)})$ for each $2\le j\le \left \lfloor \frac{n}{2}  \right \rfloor$ and each $i\in [n-1]$ by (ii) of Lemma \ref{lem3-2}, and thus $e'_j\in E(G_i^{(t)})$ for each $1\le j\le \left \lfloor \frac{n-1}{2}  \right \rfloor$ and each $i\in [n-1]$ by (6) of Observation \ref{ob1}.

If $e_1\in E(G_i^{(t)})$ for some $i\in[n-1]$, then $e_1e'_1e_2e'_2\cdots e_{\frac{n-1}{2}}e'_{\frac{n-1}{2}}$ and $e_1e'_1e_2e'_2\cdots e_{\frac{n}{2}-1}e'_{\frac{n}{2}-1}e_{\frac{n}{2}}$ are rainbow Hamiltonian paths in $\mathbf{G}^{(t)}$ for odd $n$ and even $n$, respectively, a contradiction.

Therefore, for each $i\in[n-1]$, $e_1=\{1,n\}\notin E(G_i^{(t)})$, and thus $\{k,n\}\notin E(G_i^{(t)})$ for each $k\in[n-1]$ by (6) of Observation \ref{ob1}, which implies that vertex $n$ is an isolated vertex in $G_i^{(t)}$. Since $e(G_i^{(t)})\ge \binom{n-1}{2}$ for each $i\in [n-1]$, we have $G_1^{(t)}=G_2^{(t)}=\dots=G_{n-1}^{(t)}\cong K_{n-1}\cup K_1$, say, (I) holds.

\vspace{6pt}
$\mathbf{Case~2.}$ $n=4$.
\vspace{6pt}

Since $e(G_i^{(t)})\ge \binom{n-1}{2}=3$ for each $i\in [3]$, by (6) of Observation \ref{ob1}, we have $\{1,2\}\in E(G_i^{(t)})$, otherwise, $e(G_i^{(t)})=0$, a contradiction. Furthermore, at least one of $\{1,4\}$ and $\{2,3\}$ belongs to $E(G_i^{(t)})$ for each $i\in [3]$. Otherwise, there exists some $i\in[3]$ such that $\{1,4\},\{2,3\}\notin E(G_i^{(t)})$, and then $\{3,4\},\{2,4\}\notin E(G_i^{(t)})$ by (6) of Observation \ref{ob1}, which implies $e(G_i^{(t)})\le 2$, a contradiction.

We consider the following sets:
\[
I=\{i\in[3]:\{1,4\}\in E(G_i^{(t)})\}, \qquad J=\{j\in[3]:\{2,3\}\in E(G_j^{(t)})\}.
\]
Then $I\cup J=[3]$. If $I\neq \emptyset$ and $J\neq \emptyset$, then there exist $i\in I$ and $j\in J$ such that $i\neq j$, and thus $\{4,1\}\{1,2\}\{2,3\}$ is a rainbow Hamiltonian path in $\mathbf{G}^{(t)}$, where $\{4,1\}\in E(G_i^{(t)})$, $\{1,2\}\in E(G_k^{(t)})$ with $k\in [3]\setminus\{i,j\}$ and $\{2,3\}\in E(G_j^{(t)})$, a contradiction. Thus $I=\emptyset$ or $J=\emptyset$.

If $I=\emptyset$, then $\{2,4\},\{3,4\}\notin E(G_i^{(t)})$ for each $i\in [3]$ by (6) of Observation \ref{ob1}. It follows from $e(G_i^{(t)})\ge 3$ for each $i\in [3]$ that $G_1^{(t)}=G_2^{(t)}=G_3^{(t)}\cong K_3\cup K_1$ with $d_{G_i^{(t)}}(4)=0$, (I) holds.

If $J=\emptyset$, then $\{2,4\},\{3,4\}\notin E(G_i^{(t)})$ for each $i\in [3]$ by (6) of Observation \ref{ob1}, and thus $E(G_i^{(t)})=\{\{1,2\},\{1,3\},\{1,4\}\}$, which implies $G_1^{(t)}=G_2^{(t)}=G_3^{(t)}\cong K_{1,3}$, (II) holds.

Since (I) or (II) holds, by \eqref{e5}, Lemmas \ref{lem3-4} and \ref{lem3-5}, we have $G_1=G_2=\dots=G_{n-1}\cong K_{n-1}\cup K_1$, or $n=4$ with $G_1=G_2=G_3\cong K_{1,3}$.

This completes the proof of Theorem \ref{thm1-4}.
\end{proof}

\section{Conclusion}
In this paper, we study the existence of rainbow Hamiltonian paths in graph collections under edge-size conditions. We first establish a corrected and improved version of Theorem \ref{thm1-1}, namely Theorem \ref{thm1-3}. Since Theorem \ref{thm1-3} requires at least one graph in the graph collection to be connected, we further investigate the edge-size condition guaranteeing a rainbow Hamiltonian path without connectivity restrictions and obtain Theorem \ref{thm1-4}.
In fact, using a relation between the spectral radius and the number of edges of a graph established by Stanley \cite{Stanley} in 1987, Theorem \ref{thm4-2} follows immediately from Theorem \ref{thm1-4}.

\begin{theorem}[\upshape\cite{Stanley}]\label{thm4-1}
For any graph $G$ of order $n$, we have $\rho^2(G)+\rho(G)\le2e(G)$, with equality if and only if $e(G)=\binom{k}{2}$ and $G\cong K_k\cup (n-k)K_1$ for some positive integer $k$.
\end{theorem}

\begin{theorem}[\upshape\cite{He}]\label{thm4-2}
Let $n\ge 4$ and $\mathbf{G}=\{G_1,\dots,G_{n-1}\}$ be a collection of not necessarily distinct graphs of order $n$ on the same vertex set $V$. If $\rho(G_i)\ge n-2$ for each $i\in [n-1]$, then $\mathbf{G}$ contains a rainbow Hamiltonian path unless $G_1=G_2=\dots=G_{n-1}\cong K_{n-1}\cup K_1$.
\end{theorem}

\begin{proof}
For each $i\in [n-1]$, since $\rho(G_i)\ge n-2$, by Theorem \ref{thm4-1}, we have $e(G_i)\ge \binom{n-1}{2}$. By Theorem \ref{thm1-4}, $\mathbf{G}$ contains a rainbow Hamiltonian path, or $G_1=G_2=\dots=G_{n-1}\cong K_{n-1}\cup K_1$, or $n=4$ with $G_1=G_2=G_{3}\cong K_{1,3}$.

Clearly, $\rho(K_{n-1}\cup K_1)=n-2$ and $\rho(K_{1,3})=\sqrt{3}<n-2$. Therefore, if $\rho(G_i)\ge n-2$ for each $i\in [n-1]$, then $\mathbf{G}$ contains a rainbow Hamiltonian path unless $G_1=G_2=\dots=G_{n-1}\cong K_{n-1}\cup K_1$.
\end{proof}

\section*{Funding}
This work is supported by the National Natural Science Foundation of China (Grant Nos. 12371347, 12271337).

\section*{Declarations}
\noindent\textbf{Conflict of interest} The authors declare that they have no known competing financial interests or personal relationships that could have appeared to influence the work reported in this paper.\\
\textbf{Data availability} No data was used for the research described in the article.


\section*{Appendix A. Proofs}
We present the proofs of Lemmas \ref{lem3-4} and \ref{lem3-5} here.

\subsection*{A.1. Proof of Lemma \ref{lem3-4}}
\begin{proof}
Without loss of generality, let $w\in[n]$ satisfy $d_{\mathcal{K}_{x,y}(G_i)}(w)=0$ for each $i\in[n-1]$.
Then $w\neq x$. Otherwise, for each $i\in[n-1]$, $yz\in E(\mathcal{K}_{x,y}(G_i))$ and $xz\notin E(\mathcal{K}_{x,y}(G_i))$, where $z\in [n]\setminus\{x,y\}$. This contradicts (2) of Observation \ref{ob1}.

If $w\notin\{x,y\}$, then by (3)-(5) of Observation \ref{ob1}, we readily obtain $G_i=\mathcal{K}_{x,y}(G_i)$ for each $i\in[n-1]$, and the conclusion follows.

If $w=y$, then we set $Z=[n]\setminus\{x,y\}=\{z_1,z_2,\dots,z_{n-2}\}$. By (3)-(5) of Observation \ref{ob1}, for each $i\in[n-1]$, we have $xy\notin E(G_i)$, $G_i[Z]\cong K_{n-2}$, and exactly one of $xz$ and $yz$ belongs to $E(G_i)$ for each $z\in Z$.

Now we show that $xz\in E(G_i)$ for each $z\in Z$ and each $i\in [n-1]$, or $yz\in E(G_i)$ for each $z\in Z$ and each $i\in [n-1]$. If not, there exist distinct $p,q\in[n-1]$ and distinct $z_a,z_b\in Z$ such that $xz_a\in E(G_p)$ and $yz_b\in E(G_q)$. Note that $G_i[Z]\cong K_{n-2}$ for each $i\in[n-1]$. If $n=4$, then $xz_az_by$ is a rainbow Hamiltonian path in $\mathbf{G}$. If $n\ge 5$, we take $Z\setminus\{z_a,z_b\}=\{z'_1,z'_2,\dots,z'_{n-4}\}$, and then $xz_az'_1z'_2\cdots z'_{n-4}z_by$ is a rainbow Hamiltonian path in $\mathbf{G}$. This contradicts the fact that $\mathbf{G}$ contains no rainbow Hamiltonian path.

If $xz\in E(G_i)$ for each $z\in Z$ and each $i\in [n-1]$, then $G_i\cong K_{n-1}\cup K_1$ with $G_i[Z\cup\{x\}]\cong K_{n-1}$ for each $i\in [n-1]$. If $yz\in E(G_i)$ for each $z\in Z$ and each $i\in [n-1]$, then
$G_i\cong K_{n-1}\cup K_1$ with $G_i[Z\cup\{y\}]\cong K_{n-1}$ for each $i\in [n-1]$.
Therefore, $G_1=G_2=\dots=G_{n-1}\cong K_{n-1}\cup K_1$.

This completes the proof of Lemma \ref{lem3-4}.
\end{proof}

\subsection*{A.2. Proof of Lemma \ref{lem3-5}}
\begin{proof}
Without loss of generality, let $w\in[4]$ satisfy $d_{\mathcal{K}_{x,y}(G_i)}(w)=3$ for each $i\in[3]$. Then $w\neq y$. Otherwise, $yz\in E(\mathcal{K}_{x,y}(G_i))$ and $xz\notin E(\mathcal{K}_{x,y}(G_i))$ for each $i\in[3]$, where $z\in [4]\setminus\{x,y\}$. This contradicts (2) of Observation \ref{ob1}.

If $w\notin\{x,y\}$, then $G_1=G_2=G_3=\mathcal{K}_{x,y}(G_1)\cong K_{1,3}$ by (3)-(5) of Observation \ref{ob1}, and the conclusion follows.

If $w=x$, then we set $\{z_1,z_2\}=[4]\setminus\{x,y\}$. By (3)-(5) of Observation \ref{ob1}, for each $i\in[3]$, we have $xy\in E(G_i)$, $z_1z_2\notin E(G_i)$, and exactly one of $xz$ and $yz$ belongs to $E(G_i)$ for each $z\in \{z_1,z_2\}$.

Now we show that $xz_1,xz_2\in E(G_i)$ for each $i\in [3]$, or $yz_1,yz_2\in E(G_i)$ for each $i\in [3]$. Otherwise, there exist distinct $p,q\in[3]$ and distinct $a,b\in[2]$ such that $xz_a\in E(G_p)$ and $yz_b\in E(G_q)$. It follows that $z_axyz_b$ is a rainbow Hamiltonian path in $\mathbf{G}$, a contradiction.

If $xz_1,xz_2\in E(G_i)$ for each $i\in [3]$, then $G_1=G_2=G_3=\mathcal{K}_{x,y}(G_1)\cong K_{1,3}$. If $yz_1,yz_2\in E(G_i)$ for each $i\in [3]$, then $G_i\cong K_{1,3}$ with $d_{G_i}(y)=3$ for each $i\in [3]$.

This completes the proof of Lemma \ref{lem3-5}.
\end{proof}

\end{spacing}
\end{document}